\documentclass[12pt,leqno,twoside,a4paper]{article}
\usepackage{amsfonts,amsmath,amsthm,amssymb}
\usepackage[english]{babel}
\usepackage{a4wide}
\usepackage[dvips]{graphics}
\usepackage{color}
\newcommand{\si}{\sigma}

\newcommand{\p}[2]{\phi_{{#1}^{-1}}({#2})}
\newcommand{\pn}[2]{\phi_{{#1}}({#2})}

\newcommand{\A}{A(R; S)}
\newtheorem{thm}{Theorem}[section]
\newtheorem{prop}[thm]{Proposition}
\newtheorem{lem}[thm]{Lemma}
\newtheorem{cor}[thm]{Corollary}
\theoremstyle{definition}
\newtheorem{definition}[thm]{Definition}
\newtheorem{ex}[thm]{Example}
\theoremstyle{remark}
\newtheorem{rem}[thm]{Remark}
\numberwithin{equation}{section}

\begin{document}
\title{ Finiteness conditions of $S$-Cohn-Jordan Extensions}
\author{  \bf Jerzy Matczuk\footnote{ The research was
 supported by Polish MNiSW grant No. N N201 268435} \\
 Institute of Mathematics, Warsaw University,\\
  Banacha 2, 02-097 Warsaw, Poland\\
  e.mail: jmatczuk@mimuw.edu.pl}
\date{ }
\maketitle
\begin{abstract}
 Let a  monoid $S$ act  on   a ring $R$ by injective endomorphisms
and $\A$ denote the $S$-Cohn-Jordan extension of $R$. Some
results relating finiteness conditions  of $R$ and that of $\A$ are presented.
In particular necessary and sufficient  conditions for  $\A$  to be left noetherian, to be left B\'ezout  and to be left principal ideal ring  are presented.
This also offers a solution to  Problem 10 from \cite{jm3}.
\end{abstract}
Throughout the paper $R$ stands for an associative ring with
 unity and $\phi$  denotes an action of a   multiplicative monoid  $S$  on a ring $R$ by injective
 endomorphisms.
   By this we mean that  a homomorphism  $\phi\colon S\rightarrow \mbox{End}(R)$ is
   given, such that $\phi(s)$ is an injective endomorphism of $R$, for any $s\in S$.
    It is assumed that all endomorphisms of $R$ preserve unity.

We say that an over-ring $\A$ of $R$ is an $S$-Cohn-Jordan
extension of $R$ if it is a minimal  over-ring of $R$ such that
the action of $S$ on $R$ extends to the action of $S$ on $\A$ by
automorphisms (Cf. Definition \ref{ext}).
    A classical result of  Cohn (see  Theorem 7.3.4 \cite{Cohn}) says that if the
    monoid $S$ possesses a group $S^{-1}S$ of left quotients, then
     $\A$ exists, moreover it is uniquely determined up to an
     $R$-isomorphism.
    The above mentioned theorem of Cohn was originally  formulated
    in much more general context of
    $\Omega$-algebras, not just rings. The construction of $\A$ was  given as a
    limit of a suitable directed system.
    The possibility of enlarging an object and replacing the action
    of
    endomorphisms by the action of automorphisms is a
    powerful tool, similar to  a localization.
       This is indeed the case. One can see
    \cite{L}, \cite{LM}, \cite{Mu2},  \cite{Pa},
    \cite{Pa2}, \cite{P} for examples of such applications in various
    algebraic contexts.

Jordan in \cite{Jo}   began systematic
studies of  relations between various algebraic properties of a ring $R$
and that of $\A$ in the case $S=\langle\si\rangle$ is a monoid generated by an injective endomorphism $\si$ of $R$. Then Matczuk in \cite{jm2} started such investigations in the case  $S$ is an arbitrary monoid acting, by injective endomorphisms,  on a ring $R$.  This paper can be considered as a continuation of works carried out  in  \cite{jm2}. One can also consult \cite{jm2} and  references inside for other motivations and examples of applications of $S$-Cohn-Jordan extensions.

The aim of this paper is to continue investigation of
 some finiteness  conditions of the
$S$-Cohn-Jordan extension $\A$ of $R$ in terms of properties of  $R$ and the action of $S$.   The basic idea,
which   goes back to Jordan \cite{Jo},  is to compare left ideals
$I$ of $\A$ with its  orbits    $\{  \phi(s)(I)\cap R \mid s \in
S\}$ in $R$.

The paper is organised as follows: In a short Section  1 we present technicalities needed in the next section. Those   generalize and rework
 Theorem 2.10 of \cite{jm2} which gives a correspondence between left ideals of $\A$ and certain admissible sets of left ideals of  $R$. In the same time it    simplifies considerations from \cite{jm2} and makes the paper self contained.

In Section 2 we give necessary and sufficient conditions for $\A$ to be noetherian (Theorem  \ref{thm noetherian}), left principal ideal ring (Theorem  \ref{A is PLIR}) and to be left B\'ezout ring (Proposition \ref{prop. Bezout}). Some applications and examples are presented. In particular it
appears that $\A$ is always left B\'ezout provided $R$ is such. The behaviour of the noetherian property is much more complicated. Even when $S$ is a cyclic monoid,
  one can find examples of rings $R$ and $\A$ showing that
 one of those rings is left noetherian but the other is not.

Theorem  \ref{thm noetherian} gives an  answer to Problem 10 posed in \cite{jm3}. The characterization presented   in the statement (2) of this theorem is  a generalization of the one obtained by Jordan in \cite{Jo},  in the case when $S=\langle\si\rangle$, where $\si$ is an injective endomorphism of $R$. However the ideas for his proof are different from ours.

\section{Preliminaries}Henceforth    $R$ stands for an associative   unital ring and $S$ denotes  a monoid  which possesses a group $S^{-1}S$ of left quotients.
Recall that this is the case exactly when  the monoid $S$ is left
and right cancellative and satisfies the left Ore condition. That
is, for any $s_1,s_2\in S$, there exist $t_1, t_2\in S$ such that
$t_1s_1=t_2s_2$.

Let  $\phi\colon
S\rightarrow \mbox{End}(R)$ denote  the action of  $S$ on
$R$ by injective endomorphisms.  For any $s\in S$, the
endomorphism $\phi(s)\in\mbox{End}(R)$ will be denoted by
$\phi_s$.
\begin{definition}
\label{ext}
An over-ring $\A$ of $R$ is called an $S$-Cohn-Jordan extension ($CJ$-extension, for short) of
$R$ if:
\begin{enumerate}
  \item  the action of $S$ on $R$ extends to an action of $S$ (also denoted by $\phi$) on
  $\A$ by automorphisms, i.e. $\phi_s$ is an automorphism of $\A$, for any $s\in S$.
  \item every element $a\in \A$ is of the form  $a=\phi_s^{-1}(b)$, for some
  suitable $b\in R$ and    $s\in S$.
\end{enumerate}
\end{definition}

As it was mentioned in the introduction, the $CJ$-extension $\A$ exists
and is uniquely defined up to an $R$-isomorphism (see also \cite{jm2}).

Hereafter,  as in the above definition, $\phi_s$ will also denote
the automorphism $\phi(s)$ of $\A$ and $\phi_{s^{-1}}$ will stand for its inverse $(\phi_s)^{-1}$,  where  $s\in S$. In particular, the  preimage in $R$ of a subset $X$ of $R$ under the action of $s\in S$ is equal to  $\phi_{s^{-1}}(X)\cap R$.

\begin{definition}\label{def. admisible}
A set $\{X_s\}_{s\in S}$ of subsets   of $R$ is called
$S$-admissible if,  for any $k,s\in S$, we have $R\cap
\phi_{s^{-1}}(X_{sk})=X_k$.  For such a set let $\Delta(\{X_s\}_{s\in
S})=\bigcup_{s\in S}\p{s}{X_s}\subseteq \A$.
 \end{definition}
 \begin{rem}\label{remark}
Let $\{X_s\}_{s\in S}$ be an $S$-admissible set. Then
$\pn{s}{X_k}\subseteq X_{sk}$, for any $k,s\in S$. Indeed
$\pn{s}{X_k}=\pn{s}{R\cap \phi_{s^{-1}}(X_{sk})}\subseteq
\pn{s}{R}\cap X_{sk}\subseteq X_{sk}$.
 \end{rem}
 \begin{lem} \label{X gives admisible}
Let $X$ be a subset of $\A$ and $\Gamma(X)=\{X_s=\phi_s(X)\cap
R\}_{s\in S}$. Then  $\{X_s \}_{s\in S}$ is an $S$-admissible set
of subsets  of $R$ and $X=\bigcup_{s\in S}\p{s}{X_s}$, i.e.
$\Delta\Gamma(X)=X$.
\end{lem}
\begin{proof} Let  $s, k\in S$.  Notice that
 $R\cap\p{s}{X_{sk}}=R\cap\p{s}{\pn{sk}{X}\cap R}=R\cap \pn{k}{X}\cap \p{s}{R}=R\cap
 \pn{k}{X}=X_k$,
 as $R\subseteq \p{s}{R}$. This shows that $\{X_s \}_{s\in S}$ is an $S$-admissible
 set.

 The inclusion $X\subseteq \bigcup_{s\in S}\p{s}{X_s}$ is a consequence of the fact that
  for any $x\in X$, there is $s\in S$ such that $\pn{s}{x}\in R$. The reverse inclusion holds, since $\phi_s$ is monic, for every $s\in
 S$.
\end{proof}
Notice that the set of all $S$-admissible sets  has a
natural partial ordering given by
$$\{X_s\}_{s\in S}\leq \{Y_s\}_{s\in S} \;\mbox{ if and only if }X_s\subseteq Y_s,\;\mbox{for all } s\in S.$$
\begin{prop}\label{thm iso}
There is an  order-preserving one-to-one correspondence between
the set $\mathcal{L}$ of all subsets  of $\A$ ordered by inclusion
and the partially ordered set $\mathcal{R}$ of all $S$-admissible
sets of subsets of $R$. The correspondence is given by   maps
$\Delta$ and $\Gamma$ defined above.
\end{prop}
\begin{proof} By Lemma  \ref{X gives admisible}, the maps $\Delta$ and $\Gamma$ are well-defined and satisfy $\Delta\Gamma=\mbox{id}_{\mathcal{L}}$. Clearly both maps  preserve the ordering.

Let $\{X_k\}_{k\in S}$ be  an $S$-admissible set of subsets of
$R$. Then
\begin{equation}\label{eq inclusion}
 \{X_k\}_{k\in S}\leq
\Gamma\Delta(\{X_k\}_{k\in S})=\{Y_k\}_{k\in S}
\end{equation}
$\mbox{where}\;\;Y_k=R\cap \pn{k}{\bigcup_{s\in S} \pn{s^{-1}}{X_s}}=\bigcup_{s\in S} \pn{ks^{-1}}{X_s} $.
 Let $a \in Y_k$. Then there are $s\in S$ and
$b\in X_s$ such that $a=\pn{ks^{-1}}{b}$. Since $S$ satisfies the
left Ore condition, we can pick $t,l\in S$  such that $tk=ls$.
Hence $a=\pn{t^{-1}l}{b}$ and $\pn{t}{a}=\pn{l}{b}\in
\pn{l}{X_s}\subseteq X_{ls}=X_{tk}$, where the last inclusion is
given by Remark \ref{remark}. Therefore  we obtain $a\in R\cap
\p{t}{X_{tk}}=X_{k}$, as $\{X_s\}_{s\in S}$ is an $S$-admissible
set. This shows that $Y_k\subseteq X_k$, for any $k\in S$.  This
together with (\ref{eq inclusion})  yield that $\{X_k\}_{k\in S}=
\Gamma\Delta(\{X_k\}_{k\in S})$ and complete  the proof of the
proposition.
\end{proof}

\begin{prop}\label{extensions inside}
 Let $A$ be an over-ring of $R$ such that the
  action of $S$ on $R$ extends to the action of $S$ on $A$ by automorphisms.
  Then $B=\bigcup_{s\in S}\p{s}{R}$ is a $CJ$-extension of
  $R$.
\end{prop}
\begin{proof}
Let $a,b\in
 B$ and $k,l\in S$ be such that $\pn{k}{a}, \pn{l}{b}\in R$. Since $S$ satisfies the
left Ore condition, there are $s,t, w\in S$ such that $sk=tl=w$.
 Then $\pn{w}{a}=\pn{sk}{a}, \pn{w}{b}= \pn{tl}{b}\in R$. This
 implies that $a-b, ab\in \p{w}{R}\subseteq B$
  and shows that $B$ is a subring of $A$.

  By definition of $B$,  $\p{k}{B}\subseteq B$ and $B\subseteq \pn{k}{B}$ follows, for any $k\in S$.
  The left Ore condition implies for any $k,s\in S$
  we can find $l,t\in S$ such that $ks^{-1}=t^{-1}l$. Then
  $\pn{k}{\p{s}{R}}=\p{t}{\pn{l}{R}}\subseteq \p{t}{R}$. This
  means that also  $\pn{k}{B}\subseteq B$, for $k\in S$. Now it is easy to complete the proof.
\end{proof}
We will say that a subset $X$ of $\A$ is $S$-invariant if
$\pn{s}{X}\subseteq X$, for all $s\in S$.

Direct application of Proposition \ref{extensions inside} gives
the following:
\begin{cor}\label{cor JC ext construction}
Let $T$ be an $S$-invariant subring of $R$. Then $\bigcup_{s\in
S}\p{s}{T}\subseteq \A$ is a $CJ$-extension of $T$.
\end{cor}

\begin{prop} \label{submodules}
Let $T$ be an $S$-invariant subring  of $R$
and $B=\bigcup_{s\in S}\p{s}{T}\subseteq \A$.
 Let $X$ be a subset of $\A$ and $\{X_s\}_{s\in S}=\Gamma(X)$.
 Then:
\begin{enumerate}
  \item  $X$ is an additive subgroup (a subring) of $\A$ iff for any $s\in S, $ $X_s$ is an
  additive subgroup (a subring) of $R$.
  \item $X$ is a left (right) $B$-submodule of $\A$ iff for any $s\in S, $ $X_s$ is a
  left (right) $T$-submodule of  $R$.
  \end{enumerate}
\end{prop}
\begin{proof}
(1). If $X$ is an additive subgroup (a subring) of $\A$, then so
is $X_s=\pn{s}{X}\cap R$, for any $s\in S$.

Suppose now, that   $\{X_s\}_{s\in S}$ consists of additive
subgroups (subrings) of $R$. Let $a,b\in X$. Then there are
$s,t\in S$ such that $\pn{s}{a}\in X_s$ and $\pn{t}{b}\in X_t$. By
the left Ore condition of $S$, we can pick $k,l\in S$ such that
$ks=lt=w$. Then, making use of Remark \ref{remark}, we have
$\pn{w}{a}=\pn{k}{\pn{s}{a}}, \pn{w}{b}=\pn{l}{\pn{t}{b}}\in X_w$.
Now it is easy to complete the proof of (1).

(2). We will prove only the left version of the statement (2).
Suppose that $X$ is a left $B$-submodule of $\A$ and let $s\in S$.
Then $TX_s\subseteq R\cap B\pn{s}{X}=R\cap\pn{s}{BX}\subseteq X$,
as $B=\pn{s}{B}$. This together with (1) show that $X_s$ is a
left $T$-submodule of $R$.

Suppose now, that   $\{X_s\}_{s\in S}$ consists of left
$T$-submodules of $R$. Let $b\in B$ and $x\in X$. Then there exist
$s,t\in S$ be such that $\pn{s}{b}\in T $, $ \pn{t}{x}\in X_t$.
Since $T$ is $S$-invariant, similarly as in the proof of (1), we
can find $w\in S$ such that $\pn{w}{b}\in T$ and $\pn{w}{x}\in
X_w$. Then $\pn{w}{bx}\in X_w$ and $bx\in \p{w}{X_w}\subseteq X$
follows. This together with (1) completes the proof.
\end{proof}
Let $T$ be an $S$-invariant subring of $R$. We will say that an $S$-admissible  set $\{X_s\}_{s\in S}$ of subsets of $R$ is an
$S$-admissible set of left (right) $T$-modules if each $X_s$ is a left (right) $T$-module.
Propositions \ref{thm iso} and \ref{submodules} imply the following
\begin{cor}\label{correspondance cor}Let $T$ be an $S$-invariant subring  of $R$
and $B=\bigcup_{s\in S}\p{s}{T}\subseteq \A$.
 There is a one-to-one correspondence between the set of all left (right)
 $B$-submodules of $\A$ and the set of all $S$-admissible sets of left (right)
 $T$-submodules of $R$.
\end{cor}
\begin{rem}\label{remark correspondance}
1.  If we take  $T=R$ in the above corollary, then $B=\A$ and
  the corollary gives one-to-one correspondence between the set of all  left, right, two-sided
  ideals of $\A$ and the set of all $S$ admissible sets of all left, right, two-sided
  ideals of $R$, respectively.

  2. Let  $W, T$ be $S$-invariant  subrings of $R$ such that
   $ \bigcup_{s\in S}\p{s}{W} =\bigcup_{s\in S}\p{s}{T} = B\subseteq
   \A$ (for example assume $S$ is commutative
   and take $W=R$ and $T=\pn{t}{R}$, for some $t\in S$). Then an
   $S$-admissible set $\{X_s\}_{s\in S}$ consists of
  left $W$-submodules iff it consists of
    left $T$-submodules as it corresponds to a $B$-submodule of
    $\A$. On the other hand,  observe   that
       $T$ is a left $T$-module and it does not have to be a left
       $W$-module.
\end{rem}

\begin{lem}\label{Ck}
Let $T$ be an $S$-invariant subring  of $R$, $B=\bigcup_{s\in
S}\p{s}{T}$ its CJ-extension of $T$  contained in $\A$. Then, for any subset $X$ of $R$ and $k\in S$ we have
$B\pn{k}{X}\cap R= \bigcup_{s\in S}\p{s}{T\pn{sk}{X}}\cap R$.
\end{lem}
\begin{proof}
 Let $x\in B\pn{k}{X}\cap R$. Then $x=\sum_{i=1}^nb_i\pn{k}{x_i}\in R$,
  where  $b_i\in B$ and $x_i\in X$, for $1\le i\le n$. Let $s\in
  S$  be such that $\pn{s}{b_i}\in T$, for all $1\le i\le n$. Then
  $\pn{s}{x}=\sum_{i=1}^n \pn{s}{b_i}\pn{sk}{x_i}\in T\pn{sk}{X}$.
This shows that $B\pn{k}{X}\cap R\subseteq  \bigcup_{s\in
S}\p{s}{T\pn{sk}{X}}\cap R$. The reverse inclusion is clear as,
for any $s\in S$, we have $\p{s}{T}\subseteq B$ and
$\p{s}{\pn{sk}{X}}\subseteq \pn{k}{X}$.
\end{proof}
\begin{definition}\label{def. closed ideals}
Let $T, W$ be $S$-invariant subrings of $R$.  For any
$(T,W)$-subbimodule $M$ of $R$ and $k\in S$ we define
$c^{(T,W)}_{k}(M)=\bigcup_{s\in S}\p{s}{T\pn{sk}{M}W}\cap R$.
\end{definition}

\begin{prop}\label{BMC}
Let $M$ be a  $(T,W)$-subbimodule
 of $R$, where $T, W$ are $S$-invariant subrings of $R$ and $B=A(T;S),
C=A(W,S)\subseteq \A$.   Then $\{c^{(T,W)}_{s}(M)\}_{s\in S}$ is
an admissible set of $(T,W)$-bimodules associated to the
$(B,C)$-subbimodule $BMC$ of $\A$.
\end{prop}
\begin{proof} Let us consider $(B,C)$-subbimodule $BMC$ of $\A$.
Since $\pn{s}{B}=B$ and $\pn{s}{C}$, for all $s\in S$, we have
$\Gamma(BMC)=\{B\pn{s}{M}C\cap R\}_{s\in S}$.
 Now, the proof is a direct consequence of a
bimodule versions of Corollary \ref{correspondance cor} and Lemma
\ref{Ck}.
\end{proof}
\begin{definition}\label{def. of closed} Let $M$ be a $(T,W)$-subbimodule
 of $R$, where $T, W$ are $S$-invariant subrings of $R$ and $B=A(T;S),
C=A(W,S)\subseteq \A$.  We say that $M$ is $(T,W)$-closed if
$M=BMC\cap R$.
\end{definition}

The following proposition offers an internal (in $R$)
characterization of   $(T,W)$-closed subbimodules of $R$.

\begin{prop}\label{inner description of closed}
 For a $(T,W)$-subbimodule $M$ of $R$ the following conditions are
 equivalent:
\begin{enumerate}
  \item  $M$ is $(T,W)$-closed.
  \item $c^{(T,W)}_{\mbox{\rm \scriptsize id}}(M)=M$
  \item $R\cap \p{s}{T\pn{s}{M}W}\subseteq M$, for any
$s\in S$.
  \end{enumerate}

\end{prop}
\begin{proof} Recall that $c^{(T,W)}_{ \mbox{\rm \scriptsize id}}(M)
=\bigcup_{s\in S}\p{s}{T\pn{s}{M}W}\cap R$. The equivalence
$(1)\Leftrightarrow (2)$ is given by Proposition \ref{BMC}. The
implication $(2)\Rightarrow (3)$ is a tautology.

The statement $(3)$ yields that $c^{(T,W)}_{\mbox{\rm \scriptsize
id}}(M)\subseteq M$ and clearly $M\subseteq c^{(T,W)}_{\mbox{\rm
\scriptsize id}}(M)$. This shows that $(3)\Rightarrow (2)$ and
completes the proof of the proposition.
\end{proof}

Let us notice that if $V$ is a  $(B,C)$-subbimodule of $\A$, then
$V\cap R$ is a  $(T,W)$-subbimodule of $R$ and $V\cap R\subseteq
B(V\cap R)C\cap R\subseteq V\cap R$, i.e. $V\cap R$ is a
$(T,W)$-closed subbimodule of $R$.
\begin{prop}\label{properies of closed} Let $T, W$ be $S$-invariant subrings of $R$. Then:
\begin{enumerate}
  \item If  $\{X_s\}_{s\in S}$ is an $S$-admissible set of
 $(T,W)$-subbimodules of $R$, then   $X_s$ is a closed
 $(T,W)$-subbimodule of $R$, for each $s\in S$.
  \item Let $T_1\subseteq T$ and $W_1\subseteq W$ be $S$-invariant
  subrings. Then any $(T,W)$-closed subbimodule $M$ of $R$ is
  closed as $(T_1,W_1)$-subbimodule.
    \end{enumerate}
    \end{prop}
\begin{proof}
 By Corollary \ref{correspondance cor}, there is
 $(B,C)$-subbimodule $V$ of $\A$ such that $X_s=\pn{s}{V}\cap R$.
 This together with the observation made just before the
 proposition, gives (1).

 The statement (2)  is an easy exercise if we use the description  (3) of closeness
 from of Proposition \ref{inner description of closed}.
\end{proof}

\section{Applications}
In this section we restrict our attention to left ideals, i.e. we
take $T=R$ and $W$ is the subring of $R$ generated by 1. In this
case, for $k\in S$,  we will write $c_k$ instead of $c_k^{(T,W)}$.
That is, by Proposition \ref{BMC},  $c_k(M)=\A\pn{k}{M}\cap R$, for any left ideal $M$ of $R$.

 Recall (Cf. Remark \ref{remark correspondance})(1)) that there
is one-to-one correspondence between left ideals   of $\A$ and
$S$-admissible sets of left ideals of $R$. If a left ideal $L$ of $\A$
corresponds to the $S$-admissible set $\{L_s\}_{s\in S}$, we will
say that $L$ is associated to $\{L_s\}_{s\in S}$ or that $\{L_s\}_{s\in
S}$ is associated to $L$.
\begin{definition}
We say that an $S$-admissible set $\{L_s\}_{s\in S}$ of left ideals
of $R$ is stable if there exists $k\in S$ such that
$c_s(L_k)=L_{sk}$, for all $s\in S$.
\end{definition}
The following proposition offers some other characterizations of
stability of $S$-admissible sets of left ideals.
\begin{prop} \label{characterization of stable sets}
 Let $\{L_s\}_{s\in S}$ be an  $S$-admissible set of left ideals
 of $R$ and $L$ be its associated left ideal of $\A$. The following
 conditions are equivalent:
\begin{enumerate}
  \item $\{L_s\}_{s\in S}$ is stable.
  \item  There exists $k\in S$ such that
  $\pn{sk}{L}=\A(\pn{sk}{L}\cap R)$, for any $s\in S$.
  \item There exists $k\in S$ such that $\pn{k}{L}=\A(\pn{k}{L}\cap R)$.
  \item There exist $k\in S$ and a left ideal $W$ of $R$ such that
  $\pn{k}{L}=\A W$.
\end{enumerate}
\end{prop}
\begin{proof}
 $(1)\Rightarrow (2)$. Suppose $\{L_s\}_{s\in S}$
 is stable, that is we can pick $k\in S$ such that
$c_s(L_k)=L_{sk}$, for all $s\in S$.
 Recall that $L_s=\pn{s}{L}\cap R$.  This means that
 $\{L_{sk}\}_{s\in S}=\{c_s(L_k)\}_{s\in S}$ is an $S$-admissible
 set of left ideals of $R$ associated to $\pn{k}{L}$. Now,  Proposition \ref{BMC} applied to $M=L_k$, yields that the left ideals $\pn{k}{L}$ and $\A L_k$ of $\A$ have the same associated $S$-admissible sets. Hence, by Proposition \ref{thm iso},  $\pn{k}{L}=\A L_k$.   Then, for any $s\in
 S$, we have
 \begin{equation*}\begin{split}
 \A (\pn{s}{\pn{k}{L}}\cap R) & \subseteq \pn{s}{\pn{k}{L}}= \pn{s}{\A (\pn{k}{L}\cap R)}\subseteq \\
  & \subseteq \A (\pn{s}{\pn{k}{L}} \cap \pn{s}{R})
  \subseteq \A (\pn{s}{\pn{k}{L}}\cap R).
 \end{split}
  \end{equation*}
 This shows
 that $\pn{sk}{L}=\A(\pn{sk}{L}\cap R)$, i.e. (2) holds.

 The implications $(2)\Rightarrow (3)\Rightarrow (4)$ are
 tautologies.

 $(4)\Rightarrow (1)$. Let $k\in S$ and the left ideal $W$ of $R$
 be such that $\pn{k}{L}=\A W$.
 Eventually replacing $W$ by $\pn{k}{L}\cap R$, we may additionally
 assume that $W=\pn{k}{L}\cap R=L_k$.
 Therefore, by Proposition \ref{BMC}, the left ideal $\pn{k}{L}$
 of $\A$ is associated to the $S$-admissible set $\{c_s(L_k)\}_{s\in
 S}$. Also, by definition,  $\pn{k}{L}$ is associated
 to $\{\pn{s}{\pn{k}{L}}\cap R\}_{s\in S}=\{L_{sk}\}_{s\in S}$.
 This shows that $c_s(L_k)=L_{sk}$, for any $s\in S$ and completes
 the proof of the implication.
\end{proof}
\begin{cor}\label{cor stable ideals}
 Suppose that the  $S$-admissible set $\{L_s\}_{s\in S}$ of left
 ideals of $R$   is associated to a finitely generated left ideal
 of $\A$. Then $\{L_s\}_{s\in S}$ is stable.
\end{cor}
\begin{proof}
 Let  $L=\A a_1+\ldots +\A a_n$ be a left ideal of $\A$
 associated to $\{L_s\}_{s\in S}$ and $k\in S$ be such that $\pn{k}{a_i}=b_i\in
 R$, for $1\le i\le n$. Then $\pn{k}{L}=\A W$, where
 $W=\sum_{i=1}^nRb_i$. Thus the condition (4) of
   Proposition \ref{characterization of stable
 sets} holds, i.e. $\{L_s\}_{s\in S}$ is stable.
\end{proof}

Recall (Cf. Definition \ref{def. of closed}) that a left ideal  $X$ of $R$ is closed if $X=\A X\cap R$
and that $\A X\cap R$ is always a closed left ideal of $R$. This
implies that $\A X\cap R$ is the smallest closed left ideal of $R$
containing $X$. We will call it the closure of $X$ and denote by
$\overline{X}$. Proposition \ref{inner description of closed}
offers an internal characterization of the closure of $X$, namely
$\overline{X} =\bigcup_{s\in S}\p{s}{R\pn{s}{X}}\cap R$.

With all the above preparation  we are ready to prove the following theorem.

\begin{thm}\label{thm noetherian}
 For the $CJ$-extension $\A$  of $R$ the following conditions are equivalent:

\begin{enumerate}
  \item $\A$ is left noetherian;
  \item The ring
 $R$ has \mbox{\rm ACC} on closed left ideals and every
 $S$-admissible set of left ideals is stable;
 \item   Every closed left ideal   of $R$ is the  closure of a finitely
 generated left ideal of $R$ and every $S$-admissible set of
 left ideals is stable.
 \end{enumerate}
  \end{thm}
\begin{proof}
 $(1)\Rightarrow (2)$. Suppose $\A$ is left noetherian.
 Let $X_1\subseteq X_2\subseteq \ldots$
 be a chain of closed left ideals of $R$. Since $\A$ is left noetherian,
 there exists $n\geq 1$ such that $\A X_n=\A X_{n+m}$, for all $m\ge
 0$. By assumption, every $X_i$'s is closed,
 so $X_n=\A X_n\cap R=\A X_{n+m}\cap R=X_{n+m}$, for all $m\ge 0$.
 This shows that  $R$ has \mbox{\rm ACC} on closed left ideals.

 Since $\A$ is left noetherian, every $S$-admissible set $\{L_s\}_{s\in S}$ of left ideals is
 associated to a finitely generated left ideal of $\A$. Hence, by
 Corollary \ref{cor stable ideals}, $\{L_s\}_{s\in S}$ is stable.

$(2)\Rightarrow (3)$. The proof is a version of a standard
argument. Let $W$ be a closed left ideal of $R$. Consider the set
$\mathfrak{W}$ of all closures $\overline{I}$, where $I$ ranges
over all finitely generated left ideals $I$ of $R$ contained in
$W$. Notice that if $\overline{I}\in\mathfrak{W}$ and $b\in W$,
then $\overline{I+Rb}\subseteq \overline{W}=W$. Since $R$
satisfies ACC on closed left ideals, we can pick a maximal element
$\overline{M}$ in $\mathfrak{W}$ and the  remark above yields
$W=\overline{M}$.

$(3)\Rightarrow (1)$. Let $L$ be a left ideal of $\A$ and
$\{L_s\}_{s\in S}$ its $S$-admissible set of left ideals of $R$.
By assumption, $\{L_s\}_{s\in S}$ is stable. Thus, by Proposition
\ref{characterization of stable sets},  there exist $k\in S$ and a
left ideal $W$ of $R$ such that $\pn{k}{L}=\A W$. Replacing $W$ by
$\overline{ W}$ we may additionally suppose that $W$ is closed.
Then, by assumption, there exist $b_1,\ldots, b_n\in R$ such that
$W=\overline{Rb_1+\ldots+Rb_n}$. Notice that $\A b_1+\ldots +\A
b_n\subseteq \pn{k}{L}=\A ( R\cap (\A b_1+\ldots+\A b_n))\subseteq
\A b_1+\ldots +\A b_n$. This shows that $\pn{k}{L}$ is a finitely
generated left ideal of $\A$. Since $\phi_k$ is an automorphism of
$\A$, $L$ is also finitely generated.
\end{proof}
The above theorem gives immediately:
\begin{cor} Suppose that   $R$ left noetherian. Then $\A$ is left noetherian  iff
  every $S$-admissible set of
 left ideals of $R$ is stable.
\end{cor}

The equivalence $(1)\Leftrightarrow (2)$ in Theorem \ref{thm noetherian} is a generalization of Theorem 5.6 \cite{Jo} from the case when the monoid  $S$ is cyclic to the case when $S$ is a cancellative monoid satisfying the left Ore condition. The idea of the presented proof is completely different from the one used in \cite{Jo}.

 It is know  that there exist rings $R$ such that only one of $R$ and $\A$ is  left noetherian.
 The following example, which offers such rings,  is a  variation of  examples from \cite{Jo}.
\begin{ex}1. Let $\si$ be the endomorphism of the  polynomial ring $ \mathbb{Z}[x]$ given by $\si(x)=2x$. One can check that $A(\mathbb{Z}[x];\langle\si\rangle)=\mathbb{Z}+\mathbb{Z}[\frac{1}{2}][x]x$ is not noetherian.\\
 2. Let $A$ denote the field of rational functions in the set $\{x_i\}_{i\in\mathbb{Z}}  $ of indeterminates over a field $F$ and $\si$ be the $F$-endomorphism of $R=F(x_i\mid i\leq 0)[x_i\mid i> 0]$ given by $\si(x_i)=x_{i-1}$, for $i\in \mathbb{Z}$. Then $R$ is not noetherian and $A=A(R;\langle\si\rangle)$ is a field.
\end{ex}

The following theorem offers necessary and sufficient conditions for $\A$ to be left principal ideal ring.
\begin{thm}\label{A is PLIR}
 For the $CJ$-extension $\A$  of $R$ the following conditions are
equivalent:
\begin{enumerate}
  \item Every left ideal of $\A$ is principal;
  \item Every     $S$-admissible set  $\{L_s\}_{s\in S}$ of left ideals of
  $R$ satisfies the following conditions:\\
  (a)  $\{L_s\}_{s\in S}$ is stable,\\
  (b) There exist $t\in S$ and $b\in R$ such that
  $L_t=\overline{Rb}$.
 \end{enumerate}
  \end{thm}
  \begin{proof}
  $(1)\Rightarrow (2)$.  Let $\{L_s\}_{s\in S}$ be an $S$-admissible
  set of left ideals of  $R$ and $L$ be its associated left ideal of $\A$.
   Since every left
  ideal of $\A$ is principal,
   Corollary \ref{cor stable ideals} implies that the property (a)
  holds.

  Let $a\in \A$ and $t\in S$ be such that $L=\A a$ and $ b=\pn{t}{a}\in
  R$. Then $L_t=\pn{t}{L}\cap R=\A b\cap R=\overline{Rb}$,
  i.e. the property (b) is satisfied.

$(2)\Rightarrow (1)$. Let $L$ be a left ideal of $\A$ and
$\{L_s\}_{s\in S}$ be its associated $S$-admissible
  set of left ideals of  $R$. By assumption, $\{L_s\}_{s\in S}$ is
  stable. Thus, applying   Proposition \ref{characterization of stable
  sets}(2),  we can pick $k\in S$ such that $\pn{sk}{L}=\A L_{sk}$,
  for any $s\in S$. Observe that
  $\{L_{sk}\}_{s\in S}=\{\pn{s}{\pn{k}{L}}\cap R\}_{s\in S}$ is an
  $S$-admissible set of left ideals associated to $\pn{k}{L}$.
  Therefore  we can apply  (2)(b) to $\{L_{sk}\}_{s\in S}$ and pick $l\in S$ and $b\in R$ such
  that $L_{lk}=\overline{Rb}$. Let us set $t=lk$. Using the above
  we have $\pn{t}{L}=\A L_t$ and
  $\A b\subseteq \A L_t =\A\overline{Rb}\subseteq \A b$. This
  shows that $\pn{t}{L}=\A b$ and proves that the left ideal $L=\A\p{s}{b}$ is
  principal.
  \end{proof}
\begin{rem} \label{rem assumptions}

1. It is not difficult to  prove that
 the condition  (2)(b) of the above theorem is equivalent to the
 condition that every closed left ideal $X$ of $R$ is of the form
 $X=\p{t}{\overline{Rb}}\cap R$, for suitable $t\in S$ and $b\in
 R$.\\
 2. Let us remark that the condition (2)(b) always holds, provided
 every closed left ideal is principal.
\end{rem}

Recall that a ring $R$ is left B\'ezout if every finitely generated left ideal of $R$ is principal.
\begin{prop}\label{prop. Bezout} For the $CJ$-extension $\A$  of $R$ the following conditions are
equivalent:
\begin{enumerate}
\item $\A$ is a left B\'ezout ring;
\item for every    $S$-admissible  set
 $\{L_s\}_{s\in S}$ associated to a finitely generated left
  ideal $L$ of $\A$,   there exist $t\in S$  and $b\in R$ such
  that $L_t=\overline{Rb}$.
\end{enumerate}
\end{prop}
\begin{proof}
 Let $L$ be a finitely generated left ideal of $\A$ and $\{L_s\}_{s\in
 S}$ its associated $S$-admissible set.

 If $\A$ is left B\'ezout, then $L$ is principal. Thus there is $t\in
 S$ and $b\in R$ such that $\pn{t}{L}=\A b$ and $L_t=\pn{t}{L}\cap R=\overline{R
 b}$. This shows that $(1)$ implies $(2)$.

 Suppose $(2)$ holds. Then, by Corollary \ref{cor stable
 ideals}, $\{L_s\}_{s\in S}$ is stable.

  Now one can complete the
 proof as in the proof of implication $(3)\Rightarrow(1)$ of
 Theorem \ref{A is PLIR}.
\end{proof}
Notice that the characterization obtained in the above proposition is not nice
 in the sense that the statement (2) is not expressed
in terms of properties of $R$ but $\A$ is involved. Anyway it has the following direct application:
\begin{cor}\label{cor Bezout}
Suppose that one of the following conditions is satisfied:\\
1. Every closed left ideal of $R$ is principal.\\
2. $R$ is left B\'ezout.\\
Then $\A$ is  a left B\'ezout ring.
\end{cor}

\begin{proof}
 Proposition \ref{prop. Bezout} and
 Remark \ref{rem assumptions}(2) give the thesis when (1) holds.

Suppose (2) holds. Let $L=\A a_1+\ldots \A a_n$ and $t\in S$ be
such that $b_i=\pn{t}{a_i}\in R$, $1\le i\le n$. By assumption,
there exists $b\in R$ such that $R b_1+\ldots R b_n=Rb$. Then
$L_t=\pn{t}{L}\cap R=\A b\cap R=\overline{Rb}$ and the thesis is a
consequence of Proposition \ref{prop. Bezout}.
\end{proof}

The following example offers  a principal ideal domain $R$ such that $\A$ is not noetherian.  Of course, by  Corollary \ref{cor Bezout}, $\A$ is left B\'ezout.
\begin{ex}\label{Bezout example}
 Let $A= K[x^\frac{1}{2^n}\mid  n\in \mathbb{N}]$, where $K$ is a field,  and $\si$
 be a $K$-linear automorphism of $A$ defined by $\si(x)=x^2$. Then
 the restriction of $\si$ to
 $R=K[x]$  is an endomorphism of $R$ and it is easy to check
  that $A$ is a $CJ$-extension of $R$ with respect to the action of $\si$. Notice that $A$  is
 not noetherian but it is B\'ezout, by the above corollary.
\end{ex}
In view of Theorem \ref{A is PLIR} and Proposition \ref{prop. Bezout}  it seems  interesting to know when all principal left ideals of $R$ are closed. We will concentrate on this problem till the end of the paper.
It is known (Cf. Lemma 1.16 and Theorem 2.24 of \cite{jm2}) that if $R$ is a semiprime left Goldie
ring,  then:\\
(i) every regular element $c$ of $R$ is regular in $\A$;\\
 (ii) $\A$ is   a semiprime left Goldie ring and
$Q(\A)=A(Q(R);S)$, where $Q(B)$ denotes the classical left
quotient ring of a left Goldie ring $B$.

Therefore both $Q(R)$ and $\A$ are over-rings of $R$ included in $A(Q(R);S)$.
Keeping the above notation we have:

\begin{prop}\label{basic} For a semiprime left Goldie ring $R$, the following conditions are equivalent:
\begin{enumerate}
  \item  $Q(R)\cap \A=R$;
  \item  $Rc=\overline{Rc}$ and $cR=\overline{cR}$, for every regular element $c\in
  R$;
  \item $cR=\overline{cR}$, for every regular element $c\in R$;
   \item  If $ca\in
  R$, then $a\in R$, provided  $a\in \A$ and $c\in R$ is
  regular.
\end{enumerate}\end{prop}
 \begin{proof} Let $c\in R$ be a regular element.

 $(1)\Rightarrow (2)$
 Let $a\in\A$ be such that
$ac=r\in R$. Then $a=rc^{-1}\in Q(R)\cap \A=R$. This shows that
$\A c\cap R\subseteq Rc$ and implies that $Rc=\overline{Rc}$. A
similar argument works for showing that $cR=\overline{cR}$.

The implication  $(2)\Rightarrow (3)$ is a tautology.

$(3)\Rightarrow (4)$  Suppose $ca\in R$, where $a\in \A$.
By $(3)$ we have $cR=\overline{cR}=c\A\cap R$. thus there exists $r\in R$ such that $ca=cr$ and $a=r\in R$ follows,  as $c$ is regular in $\A$.

$(4)\Rightarrow (1)$ Let $r\in R$ be such that $c^{-1}r=a\in
Q(R)\cap \A$. The condition (4) gives   $a\in R$ and shows that
$Q(R)\cap \A=R$.
 \end{proof}

The statement (2) in the above proposition is left-right symmetric  thus, additionally assuming that the semiprime ring $R$ is also right Goldie, we can add to the proposition left versions of statements (3) and (4). However, as the following example shows, we can not do this when $R$ is not right Goldie.

\begin{ex} Let $D$ denote the   field
of fractions of the ring $K[x^\frac{1}{2^n}\mid  n\in \mathbb{N}]$
from Example \ref{Bezout example}  and $\sigma$ be a $K$-linear
automorphism of $D$ defined  by $\sigma(x)=x^2$. Let us consider
the skew polynomial ring of endomorphism type (with coefficients
written on the left) $A=D[t;\sigma]$. Then $\si$ can be extended
to an automorphism of $A$ by setting $\si(t)=t$. Let
$R=K(x)[t;\si]\subseteq A$. Then the restriction of $\si $
to $R$ is an endomorphism of $R$ and for any $w\in A$, there
exists $n\ge 1$ such that $\si^n(w)\in R$. This means that
$A=A(R;\langle\si\rangle)$, where $\langle\si\rangle$ denotes the
monoid generated by $\si$.

It is well known that $R$ is a left Ore domain which is not right
Ore.  Observe that  $t\sqrt{x}=xt\in R$, but $\sqrt{x}\not\in R$.
Thus, by Proposition \ref{basic}, $Q(R)\cap A\ne R$. In fact,  the left localization of $R$ with respect the left Ore
set consisting of all powers of $t$ is equal to $D[t,t^{-1},\si]$. Thus  $A\subseteq D[t,t^{-1},\si]
\subseteq Q(R)$.

We claim that $R$ satisfies the left version of statement
(4) from Proposition \ref{basic}. Let $0\ne c\in R$ and $a\in A$ be such that
$ac\in R$.  If  $a\not \in R$, then we
can choose such $a=\sum_{i=0}^na_it^i$ of minimal possible degree,
say $n$. Then, by the choice of $n$,  $a_n \not \in K(x)$. Then
also $a_n\si^n(c_m)\not \in K(x)$, where $c_m$ denotes the leading
coefficient of $c\in R=K(x)[t;\si]$ and then $ac\not \in R$, which is impossible. Thus $a$ has to belong to $ R$.
\end{ex}

 Observe that the ring from Example \ref{Bezout example} satisfies the
assumption of the following proposition.
\begin{prop}
 Suppose $R$ is a left Ore domain such that $Q(R)\cap \A=R$. For a left ideal  $L$ of $A$ the following
 conditions are equivalent:
\begin{enumerate}
  \item  $L$ is principal;
  \item  $\exists_{s\in S}\exists_{a\in R}$
  such that, for all $t\in S$,  $\phi_{ts}(L)\cap R=L_{ts}=R\phi_t(a)$.
\end{enumerate}
\begin{proof}
 $(1)\Rightarrow (2).$ Suppose $L=Ab$ and let $s\in S$ be such that $\phi_s(a)\in R$.
 Then $\phi_s(L)=A\phi_s(b)$. Set $a=\phi_s(b)$. Now the implication is a direct
 consequence of Proposition \ref{basic}.

 $(2)\Rightarrow (1).$ Let $s\in S$ and $a\in R$ be as in $(2)$.
 Then $\phi_s(L)_t=L_{ts}R\phi_t(a)$. This means
 that $\phi_s(L)=\bigcup_{t\in S}\phi_t^{-1}(R)a=\A a$. Thus
 $L=\A\phi_s^{-1}(a)$.
\end{proof}
\end{prop}

\end{document}